\documentclass[10pt]{article}

\usepackage{amsmath,amssymb,amscd,amsthm,makeidx,graphicx,url,bm,bbm,mathrsfs,palatino,bm}
\usepackage{a4wide,xy}
\usepackage{enumerate,comment}
\usepackage{titlesec}
\titleformat{\paragraph}[runin]
{\normalfont\normalsize\bfseries}{\theparagraph}{1em}{}
\titlespacing*{\paragraph}
{0pt}{3.25ex plus 1ex minus .2ex}{1.5ex plus .2ex}

\usepackage[
breaklinks=true,
pdftitle={Mirror symmetry with branes via equivariant Verlinde formulae},
pdfauthor={Tam\'as Hausel, Du Pei}]{hyperref}
\usepackage{color}
\definecolor{tocolor}{rgb}{.1,.1,.5}
\definecolor{urlcolor}{rgb}{.2,.2,.6}
\definecolor{linkcolor}{rgb}{.1,.4,.6}
\definecolor{citecolor}{rgb}{.6,.3,.1}
\hypersetup{colorlinks=true, urlcolor=urlcolor, linkcolor=linkcolor, citecolor=citecolor}

\let\oldmarginpar\marginpar
\renewcommand\marginpar[1]{\-\oldmarginpar[\raggedleft\footnotesize #1]
{\raggedright\footnotesize #1}}

\makeindex

\xyoption{all}
\input{xypic}

\numberwithin{equation}{section}

\newtheorem{theorem}[equation]{Theorem}
\newtheorem{proposition}[equation]{Proposition}
\newtheorem{corollary}[equation]{Corollary}

\newtheorem{lemma}[equation]{Lemma}

\theoremstyle{remark}
\newtheorem{remark}[equation]{Remark}
\newtheorem{example}[equation]{Example}

\theoremstyle{definition}
\newtheorem{definition}[equation]{Definition}

\newcounter{margin}
{\end{itshape}  \bigskip}

\def\beq{\begin{eqnarray}}
\def\eeq{\end{eqnarray}}
\def\bes{\begin{eqnarray*}}
\def\ees{\end{eqnarray*}}

\DeclareMathOperator{\Pic}{Pic}

\DeclareMathOperator{\Hom}{Hom}
\DeclareMathOperator{\Ext}{Ext }
\DeclareMathOperator{\Tr}{Tr} 

\DeclareMathOperator{\Res}{Res} 
\DeclareMathOperator{\rank}{rank}

\def\Jac{{\mathcal J}}
\def\J{\Jac}
\def\bi{{\overline{i}}}
\def\pr{{\rm pr}}

\def\rs{{C}}
\def\bV{{\bf{V}}}
\def\bE{{\bf{E}}}
\def\bPhi{{\bf{\Phi}}}
\def\bphi{{\bm{\phi}}}

\def\C{\mathbb{C}}
\def\M{{\mathcal{M}}}

\def\calF{{\mathcal{F}}}

\def\calP{\mathcal{P}}
\def\calH{\mathcal{H}}

\def\N{\mathbb{Z}_{\geq 0}}

\def\R{\mathbb{R}}

\def\Q{\mathbb{Q}}

\def\calO{{\mathcal O}}

\def\Z{\mathbb{Z}}

\newcommand{\nc}{\newcommand}

\def\ch{{\rm ch}}
\def\td{{\rm td}}

\usepackage{leftidx}

\nc{\op}[1]{\mathop{\mathchoice{\mbox{\rm #1}}{\mbox{\rm #1}}
{\mbox{\rm \scriptsize #1}}{\mbox{\rm \tiny #1}}}\nolimits}
\nc{\al}{\alpha}

\nc{\ep}{\varepsilon} 
\nc{\ga}{\gamma} 
\nc{\Ga}{\Gamma}
\nc{\la}{\lambda} 
\nc{\La}{\Lambda} 
\nc{\si}{\sigma}
\nc{\Sig}{{\Gamma}} 
\nc{\Om}{\Omega} 
\nc{\om}{\omega}

\nc{\SL}{\mathrm{SL}} 
\nc{\GL}{\mathrm{GL}} 
\nc{\PGL}{\mathrm{PGL}}
\nc{\G}{\mathrm{G}}

\nc{\W}{\mathrm{W}}
\nc{\Lg}{\mathrm{L}}
\nc{\Pg}{\mathrm{P}}
\nc{\T}{{\mathbb T}}
\nc{\calL}{{\mathcal L}}
\nc{\Sym}{{\rm Sym}}

\newcommand{\ft}{{\frak t}}

\newcommand{\fh}{{\frak h}}

\newcommand{\fm}{{\frak m}}
\newcommand{\vt}{{\vartheta}}
\newcommand{\ra}{{\rightarrow}}
\renewcommand{\H}{\mathrm{H}}

\nc{\Frob}{\mathrm{Frob}}

\def\U{{\mathrm{U}}}
\def\SU{{\mathrm{SU}}}

\nc{\rN}{\mathrm{N}}

\usepackage{longtable}

\nc{\cpt}{{\op{cpt}}} \nc{\Dol}{{\op{Dol}}} \nc{\DR}{{\op{DR}}}
\nc{\B}{{\op{B}}} \nc{\Triv}{\op{Triv}} \nc{\Hod}{{\op{Hod}}}
\nc{\Log}{{\op{Log}}} \nc{\Exp}{{\op{Exp}}} \nc{\Est}{E_{\op{st}}}
\nc{\Hst}{H_{\op{st}}} \nc{\Left}[1]{\hbox{$\left#1\vbox to
10.5pt{}\right.\nulldelimiterspace=0pt \mathsurround=0pt$}}
\nc{\Right}[1]{\hbox{$\left.\vbox to
10.5pt{}\right#1\nulldelimiterspace=0pt \mathsurround=0pt$}}
\nc{\LEFT}[1]{\hbox{$\left#1\vbox to
15.5pt{}\right.\nulldelimiterspace=0pt \mathsurround=0pt$}}
\nc{\RIGHT}[1]{\hbox{$\left.\vbox to
15.5pt{}\right#1\nulldelimiterspace=0pt \mathsurround=0pt$}}

\nc{\bee}{{\bf E}} 

\begin{document}

\title{Mirror symmetry with branes by equivariant Verlinde formulae}

\author{ Tam\'as Hausel
\\ {\it IST Austria} 
\\{\tt tamas.hausel@ist.ac.at} \and Anton Mellit\\ {\it University of Vienna} 
\\{\tt anton.mellit@univie.ac.at }  \and Du Pei \\ {\it QGM, Aarhus University} \\ {\it Walter Burke Institute for Theoretical Physics, Caltech}
\\{\tt pei@caltech.edu}  }
\pagestyle{myheadings}

\maketitle
\begin{abstract} We find an agreement of equivariant indices of semi-classical homomorphisms  between pairwise mirror branes in the $\GL_2$ Higgs moduli space on a Riemann surface. On one side we have the components of the Lagrangian brane of $\U(1,1)$ Higgs bundles whose mirror was proposed by Nigel Hitchin to be certain even exterior powers of the hyperholomorphic Dirac bundle on the $\SL_2$ Higgs moduli space. The agreement arises from a mysterious functional equation. This gives strong computational evidence for Hitchin's proposal. 
\end{abstract}		
\section{Introduction} In the conference "Hitchin 70: Differential Geometry and Quantization" in Aarhus in September 2016, the first and third author gave talks on recent results on the understanding of equivariant Verlinde formulae on Higgs moduli spaces. Here we explain how such techniques can be used to compute equivariant indices of semi-classical homomorphisms between branes on Higgs moduli spaces, and how in turn this could be used to give non-trivial computational evidence for mirror symmetry proposals of Hitchin in \cite[\S 7]{hitchin}. 

 Hitchin introduced the moduli space of Higgs bundles on a Riemann surface \cite{hitchin-self} in 1987. 
In 2006, the work of Kapustin-Witten \cite{kapustin-witten} proposed an  understanding of the Geometric Langlands program using the $S$-duality of 4D supersymmetric Yang-Mills theory. One of the main statements is that $S$-duality induces mirror symmetry between the moduli space $\M_\DR(\G)$  of flat $\G$-connections on $C$ and $\M_\DR(\G^L)$ for the Langlands dual group $\G^L$. The precise statement is still missing, but we have more understanding \cite{donagi-pantev} in the semi-classical limit where the mirror symmetry should reduce to a relative Fourier-Mukai type equivalence between $D(\M(\G))$ and			$D(\M(\G^L))$, where $\M(\G)$ denotes certain moduli space of $\G$-Higgs bundles on $C$.  Even this duality is not completely understood as we lack the description of the Fourier-Mukai transform on the most singular fibers of the Hitchin map. Consequently, we have few global results proving aspects of this duality. 

One global computational result was achieved in \cite{hausel-thaddeus} in 2002. There, as a consequence of mirror symmetry, a conjecture was proposed for the agreement of certain Hodge numbers of $\M(\SL_n)$ and $\M(\PGL_n)$ which were proved there for $n=2,3$. This conjecture for every $n$ has just recently been settled in \cite{groechenig-etal} using $p$-adic integration. 

Here we are studying another global  approach where deeper properties of the conjectured mirror symmetry could be 	computationally verified. We set $\G=\GL_2$ and $C$ a smooth complex projective curve of genus $g>1$. The components of the $\M(\U(1,1))$ Higgs moduli space inside the $\M(\GL_2)$ moduli space can be labeled by certain characteristic classes as explained in \cite{hitchin}. There are $g$ such components and we will denote these Lagrangians by $$L_0,\dots,L_{g-1}\subset \M(\GL_2).$$ To be more precise $L_i$ is the locus of Higgs bundles of the form $$M_1\oplus M_2 \stackrel{\Phi}{\to} M_1K_C\oplus M_2K_C,$$ where $M_j$ are line bundles with $\deg(M_1)=g-1-i$, $\deg(M_2)=i-g+1$ and $\Phi$ off-diagonal. For $i=0,\dots,g-2$ the Lagrangians are located in the stable locus $\M(\GL_2)^s$ and so are smooth and $L_i$ is isomorphic to an affine 	bundle over the $2i$th symmetric product of $C$ times the Jacobian of degree $g-1-i$ line bundles on $C$. The last component $L_{g-1}$ is special as it contains strictly semi-stable points as well. 

  As is explained in \cite[\S 7]{hitchin} $L_i$ is a $\rm BAA$ brane whose mirror should be a $\rm BBB$ brane i.e. a hyperholomorphic sheaf on the mirror $\M(\GL_2)$ (cf. also \cite{baraglia-schaposnik} and \cite[\S 6]{GW}). Its support will be $\M(\SL_2)\subset \M(\GL_2)$ and roughly speaking it is  constructed from a universal bundle $\bE$ on $\M(\SL_2)\times C$ by taking $\bV:=R^1\pi_*(\bE)$ by the projection $\pi$ to the first factor. Then Hitchin proposes that the brane $L_i$ should be mirror to $\Lambda^{2i}\bV$. This was checked to be correct on a generic fiber $\chi_a:=\chi^{-1}(a)$  of the Hitchin map for $a\in \calH_{\SL_2}$ in the sense that ${\cal{O}}_{L_i} |_{\chi_a}$ is Fourier-Mukai dual to $\Lambda^{2i}\bV |_{\chi_a}$. Additionally, the fact that $\Lambda^{2i}\bV$ carries a natural hyperholomorphic structure gave more credence to the proposal that $\Lambda^{2i}\bV$ should be the mirror of $L_i$. However it is unclear how to extend the relative Fourier-Mukai transform to the most singular fibers of the Hitchin map, in particular to the nilpotent cone $\chi_0$, and so we did not have direct checks of the correctness to the proposal that $L_i$ is indeed relative Fourier-Mukai dual to $\Lambda^{2i}\bV$.  

There have recently been several works computing equivariant Verlinde formula on the moduli space of Higgs bundles. Let $\calL$ be the determinant line bundle on $\M(\G)$ where $\G$ is a  simply-connected semisimple  group, such as $\G=\SL_n$. The natural $$\T:=\C^\times$$ action on $\M(\G)$ by scaling the Higgs field lifts to a $\T$-equivariant structure on $\calL$. Although $H^*(\M(\G);\calL^k)$ is infinite dimensional,  the induced $\T$-action will have finite dimensional weight spaces and non-trivial only for non-positive weights: $$H^*(\M(\G);\calL^k)=\bigoplus_{j=\in \N} H^*(\M(\G);\calL^k)^{-j}.$$ We define for the $\T$-equivariant line bundle $\calL^k$ (and similarly later for a $\T$-equivariant sheaf) the equivariant index as \beq \label{character} \chi_\T\left(\M(\G);\calL^k\right)=\sum_{i,j} (-1)^i\dim(H^i(\M(\G),\calL^k)^{j})t^{-j}\in \Z[[t]].\eeq   In a series of papers \cite{GP,GPYY,AGP} following insights from studying path integrals in certain supersymmetric quantum field theories and TQFT techniques, precise formulas have been achieved for \eqref{character}. Equivalent formulas have been also found in \cite{HL}. There is also a more combinatorial understanding of the results for $\G=\SL_n$ as residue formulas in a TQFT framework in \cite{hausel-szenes}.  

In this paper we compute similar characters but only on $\M(\GL_2)$ and eventually on $\M(\SL_2)$. Namely, we will consider the (generically) vector bundles $\Lambda^{2i}\bV$ mentioned above, with support in $\M(\SL_2)$ together with a $\T$-equivariant structure. We will also consider the analogues of determinant bundle $\calL$ on $\M(\GL_2)$ with a $\T$-equivariant structure and in particular the $\T$-equivariant coherent sheaves $\calL^2\calO_{L_i}$ on $\M(\GL_2)$ (for the detailed constructions see Section~\ref{background}). With the notation   $\calL_i:=\calL^2\calO_{L_i}$ for $0\leq i < g-1$ 
and $\Lambda_j:=\Lambda^{2j}\bV$  our main result is  the following

\begin{theorem}\label{mainth} For $0\leq i,j <g-1$  we have
\beq \label{mainf} \chi_\T\left(\M(\GL_2);\calL_j\otimes \Lambda_i\right)=\chi_\T\left(\M(\GL_2);\calL_i\otimes \Lambda_j\right)\in \Z[[t]].\eeq
\end{theorem}
In Corollary~\ref{mainc} we will explain how to extend this symmetry for the case of $i=g-1$. 

We can interpret the left-hand side of \eqref{mainf} as the equivariant index of derived homomorphisms from $(\Lambda^{2i} \bV)^\vee \cong \Lambda^{2i} \bV$ to $\calL^2\calO_{L^j}$ and the right-hand side  as the equivariant index of derived homomorphisms from $\calL^2\calO_{L^i}^\vee\cong\calL^2\calO_{L^i}$ to $\Lambda^{2j} \bV$. For more details on this see Section~\ref{reflect}. This way our Theorem~\ref{mainth} provides computational evidence for Hitchin's proposal in \cite[\S 7]{hitchin}. 

Interestingly, most of the ingredients used here have already appeared in the first author's PhD dissertation \cite{hausel-thesis} written under the supervision of Nigel Hitchin. In a sense our computation could have been done almost $20$ years ago. The present paper thus outlines the importance  of ideas from physics, in particular  the work of Kapustin-Witten \cite{kapustin-witten} in 2006 and further mathematical insights by Hitchin \cite[\S 7]{hitchin} in 2013.

The plan of the paper is as follows. In Section~\ref{background} we introduce the $\C^\times$-equivariant coherent sheaves $\calL_i$ and $\Lambda_i$ on $\M(\GL_2)$. In Section~\ref{computation} we compute the equivariant index $\chi_\T(\M(\GL_2);\calL_i\otimes \Lambda_j)$ for $i<g-1$, by reducing the computation via Hirzebruch-Riemann-Roch to an integral on the $2i$-th symmetric product of $C$, which we evaluate using a residue formula due to Zagier.  In Section~\ref{symmetrysmoothi} we find a change of variables which will imply our main symmetry observation for $i,j<g-1$. In Section~\ref{global} following the technique of Teleman-Woodward \cite{TW} we compute  the equivariant index on the whole moduli stack and  then in Section~\ref{symmetryg-1} we prove Corollary~\ref{mainc}. Finally, in Section~\ref{reflect} we sketch how this should be a consequence of mirror symmetry.

\noindent {\bf Acknowledgements.} We would like to thank Sergei Gukov and Paul Harmsen for encouragement and their interest in this paper,  J{\o}rgen E.~Andersen, Andr\'as Szenes, Laura Schaposnik, Nigel Hitchin, Penghui Li, Iordan Ganev and the referee for useful explanations and comments. The research in this paper was supported   by  an Advanced Grant ``Arithmetic and physics of Higgs moduli spaces'' no.\ 320593 of the European Research Council, the NCCR SwissMAP of the Swiss National Foundation,  START-Project Y963-N35 of the Austrian Science Fund (FWF) , the  center of excellence grant ``Center for Quantum Geometry of Moduli
Space'' from the Danish National Research Foundation (DNRF95), by the Walter Burke Institute
for Theoretical Physics, and by the U.S. Department of Energy, Office of Science, Office of High
Energy Physics, under Award Number DE-SC0011632. In particular, the idea of the computation in this paper arose during the 
\href{https://bernoulli.epfl.ch/events/1001}{"Retreat on Higgs bundles, real groups, Langlands duality and mirror symmetry"} in the Bernoulli center at EPF Lausanne in January 2016. 
\section{Background} \label{background}

Let $C$ be a smooth complex projective curve of genus $g>1$. Let $\M(\GL_2)$ denote the moduli space of semistable degree $0$ rank $2$ Higgs bundles on $C$. We recall \cite{bradlow-etal,hitchin,schaposnik} that the $\U(1,1)$-Higgs bundles in $\M(\GL_2)$ are degree $0$ rank $2$ Higgs bundles \beq\label{higgs}E\stackrel{\Phi}{{\to}} E\otimes K_C\eeq such that $\Phi\neq 0$ and $$E\stackrel{-\Phi}{{\to}} E\otimes K_C$$ is equivalent with \eqref{higgs}. In other words there exists an $a:E\to E$ automorphism of $E$ such that \bes\Phi\circ a=-\Phi.\ees When $\Phi\neq 0$ this implies that such a Higgs bundle has the form \beq \label{form} M_1\oplus M_2\stackrel{\left(\begin{array}{cc}0 & \phi_2 \\ \phi_1 & 0 \end{array}\right)}{\longrightarrow}  (M_1\oplus M_2)\otimes K_C\eeq is a direct sum of two line bundles and off-diagonal Higgs field. We can assume $$\deg M_2\leq \deg M_1=g-1-i$$ for some integer $i\leq g-1$. When $i<g-1$ then $$\phi_1\in H^0(C;M_1^{-1}M_2K_C)$$ cannot be $0$ because of semistability. As $$\deg(M_1^{-1}M_2K_C)  = 2i,$$ we have that $i\geq 0$.


The locus of Higgs bundles of the form \eqref{form} with $\deg(M_1)=g-1-i$ for  $0\leq i \leq g-1$ is denoted by $$L_i\subset \M(\GL_2).$$ 
With an argument similar to the one below for $\SL_2$ one can show that for $0\leq i <g-1$ the locus $L_i\subset \M(\GL_2)$ is isomorphic with a total space of a vector bundle over $\J_\bi\times C_{2i}$. We denote by $\J_i$ the Jacobian of degree $i$ line bundles on $C$, $C_j$ the $j$th symmetric product of the curve $C$ and $\bi:=g-1-i.$

We will also need to define the sheaf $\calL$. Recall the tensor product map $$\tau:\M(\SL_2)\times T^*\J_0 \to \M(\GL_2),$$ which is a Galois cover
with Galois group $\J[2]$ the two torsion points on the Jacobian. We define $\calL\in \Pic(\M(\GL_2))$ such that $$\tau^*(\calL)=\calL_{\SL_2}\boxtimes \calO_{T^*\J_0},$$ where $\calL_{\SL_2}$ is the determinant line bundle on $\M(\SL_2)$ constructed below \eqref{detsl2}. $\calL$ is not unique, because $$\tau^*:\Pic(\M(\GL_2))\to \Pic(\M(\SL_2)\times \J_0)$$ has a kernel isomorphic to $$\Pic(\M(\GL_2))[2]\cong \J_0[2].$$ However when restricted to $\M(\U(1,1))$ the ambiguity disappears. For $0\leq i< g-1$ we will denote $$\calL_i:=\calL^2\otimes \calO_{L_i}$$ a coherent sheaf on $\M_{\GL_2}$, while for $i=g-1$ we define
\beq \label{lg-1} \calL_{g-1}:=(\calL^2\oplus \calL^2)\otimes \calO_{L_{g-1}}.\eeq

In our computation below we will restrict to the part $$L^\prime_i:=L_i\cap \M(\SL_2)\subset \M(\SL_2) $$ inside the moduli space  $\M(\SL_2)$  of semistable rank $2$ Higgs bundles with trivial determinant and trace-free Higgs field. In other words $L^\prime_i$ is a component of the moduli space of  $\SU(1,1)$-Higgs bundles in $\M(\SL_2)$.
This means to set $M:=M_1$ and choose $M_2=M^{-1}$ in  \eqref{form}.

 Now we give a more detailed description  of $L^\prime_i\subset \M(\SL_2)$ for $0\leq i \leq g-2$.  For similar discussions  see \cite[Lemma 6.1.2]{hausel-thesis} and \cite[(6.1)]{HT2}.

	Define
maps
$\Jac_{\bi}\to \Jac_{2i}$ by sending 
$M\mapsto M^{-2} K$ and the Abel-Jacobi map 
$C_{2i}\to \Jac_{2i}$ by sending $D\mapsto L(D)$, then we construct the fibred product of the these
maps: 
\begin{eqnarray} F^\prime_i:=C_{2i}\times_{\Jac_{2i}}
\J_{\bi}.\label{fibredproduct} \end{eqnarray}
We have the two projections $\pr_{\J_{\bi}}:F^\prime_i\to \J_{\bi}$ and 
$\pr_{C_{2i}}:F^\prime_i\to C_{2i}$. 
By construction	 $F^\prime_i$ is isomorphic to the moduli
space of complexes: $M\stackrel{\phi}{\to}M^{-1} K_C$, with
$\deg(M)=\bi$ and $\phi\in
H^0(\Sigma;M^{-2} K_C),$ in other words of nilpotent $\SL_2$ Higgs bundles $$M\oplus M^{-1}\stackrel{\Phi}{\to} MK_C\oplus M^{-1}K_C$$ given by $$\Phi=\left(\begin{array}{cc} 0 & 0 \\ 
\phi & 0 \end{array}\right).$$ Let $$\Delta_{2i}\subset C\times C_{2i}$$ denote the universal divisor and $\calP_{\bi}$ a normalized universal bundle on $C\times \J_{\bi}$. By abuse of notation we will denote  their pull-backs by $\pr_{C_{2i}}$ and respectively $\pr_{\J_{\bi}}$
to $C\times F^\prime_i$       with the same $\Delta_{2i}$ and $\calP_{\bi}$. As $i<g-1$ we have $$\deg(K_C^2(-D))>2g-2$$  for $D\in C_{2i}$ and so $$H^1(C;K^2_C(-D))=0.$$ Consequently $$E_i:={\pr_{C}}_*\left(K_C^2(-\Delta_{2i})\right)$$ is a vector bundle on $F^\prime_i$.           

$E_i$ parametrizes $\SU(1,1)\subset \SL_2$ Higgs bundles as follows. First note that $\calP_{\bi}^2K_C^{-1}(\Delta_{2i})$ restricted to $x\times C$ is a trivial line bundle for every $x\in F^\prime_i$ and thus by the push-pull formula there is a line bundle $M_{F^\prime_i}$ on $F^\prime_i$ such that \beq\label{basic}\calO(\Delta_{2i})\cong M_{F^\prime_i}K_C\calP^{-2}_\bi.\eeq
Consider the rank $2$ vector bundle
\beq\label{univerfi}\bE_i:=\calP_{\bi}\oplus \calP_{\bi}^{-1}M_{F^\prime_i}\eeq on $C\times F^\prime_i$ and the Higgs field $$\bPhi_i=\left(\begin{array}{cc} 0 & \bphi_i^\prime \\ 
\bphi_i & 0 \end{array}\right)$$ with $$\bphi_i\in H^0(C\times E_i;\calP_i^{-2}M_{F^\prime_i}K_C)=H^0(C\times E_i;\calO(\Delta_{2i}))$$ given by the universal divisor $\Delta_{2i}$ and $$\bphi_i^\prime\in H^0(C\times E_i;\calP_i^{2}L^{-1}_{E_i}K_C)= H^0(C\times E_i;K_C^2(-\Delta_{2i}))$$ the tautological section. 

Thus $(\bE_i,\bPhi_i)$ is a family of stable $\SL_2$ Higgs bundles parametrized by (the total space of) $E_i$ thus we get an embedding $$\iota_i:E_i\to \M(\SL_2)^{s}\subset \M(\SL_2).$$ We denote its image $$L^\prime_i:=\iota_i(E_i)\subset \M(\SL_2)$$ and note that it parametrizes all Higgs bundles of the form $M\oplus M^{-1}\stackrel{\Phi}{\to} MK_C\oplus M^{-1}K_C$  with Higgs field $$\Phi=\left(\begin{array}{cc} 0 & \phi^\prime \\ 
\phi & 0 \end{array}\right)$$ and such that $\deg M=\bi$. This precisely agrees with a component of the $\SU(1,1)$ Higgs moduli space as described in \cite{schaposnik} (see also \cite[\S 7]{hitchin}). Thus $L^\prime_i$ is a connected smooth Lagrangian subvariety in the stable locus of $\M(\SL_2)$. Note that there is one more component of $\M(\SU(1,1))$ inside $\M(\SL_2)$, namely $L^\prime_{g-1}$. However, it will no longer be contained in the stable locus and will not be smooth.

We are now going to construct the vector bundles $\Lambda_j$ on $L^\prime_i$ for $j=0,\dots,g-1$. Hitchin \cite[\S 7]{hitchin} defines $$\Lambda_j:=\Lambda^{2j}\bV$$ locally,  where $\bV$ is the Dirac bundle on $U\to \M(\SL_2)^{s}$ for an \'etale open $U$. It is 
constructed as $$\bV:=R^1\pr_{C*}(\bE_U\stackrel{\bPhi_U}{\to}\bE_U K_C)$$ where 
$$\bE_U\stackrel{\bPhi_U}{\to}\bE_U K_C$$ is a universal $\SL_2$-Higgs bundle on $C\times U$. Universal Higgs bundle does not exist on the whole $C\times \M(\SL_2)^{s}$ but it exists on an \'etale open covering, and the obstruction to glue vanishes for the even exterior powers $\Lambda_j=\Lambda^{2j}\bV$. We note that $\bV$ is a vector bundle on $\M(\SL_2)^s$ as for stable degree $0$ Higgs bundles $(E,\Phi)$ the zeroth and second
hypercohomology of $E\stackrel{\Phi}{\to} EK_C$ vanishes by \cite[Corollary 3.5]{hausel-vanishing}.

Finally, we will extend the construction of $\bV$ to the whole of $\M(
\SL_2)$ as a complex of coherent sheaves by defining it for a $U\subset \M(\SL_2)$ by the formula $$\bV:=R\pr_{C*}(\bE_U\stackrel{\bPhi_U}{\to}\bE_U K_C)[1].$$ Then we define $$\Lambda_j:=\Lambda^{2j}\bV$$ as a complex of coherent sheaves on the whole $\M(\SL_2)$.

When we restrict this construction of $\Lambda_j$ to $L^\prime_i$ for $i<g-1$ then we have almost a universal bundle on $C\times L^\prime_i$. Namely, $(\bE_i,\bPhi_i)$ is a family of stable $\SL_2$ Higgs bundles on $C\times L^\prime_i$. However $\bE_i$  is not itself an $\SL_2$  bundle, as $$\det(\bE_i)=M_{F^\prime_i}.$$ $M_{F^\prime_i}$ is pulled back from $L^\prime_i$ and ultimately from $F^\prime_i$, which we will see has no square root.
However, on the even exterior powers we can just tensor with integer
powers of $\det(\bE_i)$ to get the restriction of $\Lambda_j$ to $L^\prime_i$ and so we get:
\beq \label{restrict}\Lambda_j|_{L^\prime_i}\cong M_{F^\prime_i}^{-j}\Lambda^{2j}(R^1\pr_{L^\prime_i*}(\bE_i\stackrel{\bPhi_i}{\to}\bE_i K_C))\eeq which gives the desired vector bundle on $L^\prime_i$.      

We will also need $\T$-equivariant structure on $\Lambda_j$. To construct this we note
that Hitchin's construction can be done on $\T$-equivariant \'etale open $U\to \M(\SL_2)^s$, where from the universal property one gets a $\T$-equivariant structure
on $\bE_U\stackrel{\bPhi_U}{\to}\bE_U K_C$ (cf. \cite[\S 4]{HT1}). They will then glue to yield a $\T$-equivariant
structure on $\Lambda_j$. The $\T$-action we endow $\Lambda_j$  is the one twisted by the trivial line bundle with a weight $-j$ $\T$-action on it.  

To see the $\T$-equivariant structure on $\Lambda_j|_{L^\prime_i}$ in \eqref{restrict} we endow 
$\calP_i$ with the trivial $\T$-action and $ M_{F^\prime_i}$ with a weight minus one $\T$-action, which will give the $\T$-equivariant structure on $\bE_i$ in \eqref{univerfi}. Then $(\bE_i,\bPhi_i)$ becomes a $\T$-equivariant Higgs bundle, provided we twist the second term $\bE_i\otimes K_C$ with a weight one $\T$-equivariant trivial line bundle. The $\Lambda_j|_{L^\prime_i}$ will then inherit a $\T$-equivariant structure, which we will further twist with a weight $-j$ trivial line bundle. 

Finally, we will need to construct the $\T$-equivariant determinant line bundle $\calL_{\SL_2}$ on $\M(\SL_2)$ (which by abuse of notation we will abbreviate to $\calL$). It is again constructed \cite{quillen} by gluing together \beq\label{detsl2}\Lambda^{2g-2}(R\pr_{C*}(\bE_U)[1])\eeq from open sets $U\subset \M(\SL_2)$ in an open covering. 

Again on $L^\prime_i$ we can compute it from our family $(\bE_i,\bPhi_i)$ and get 
\beq\label{determinant}\calL|_{L^\prime_i}=M_{F^\prime_i}^{-g+1}\Lambda^{2g-2}(R\pr_{C*}(\bE_i)[1]).\eeq

We can now define our other family of $\T$-equivariant sheaves \beq\label{calli}\calL^\prime_i:=\iota_{i*} \calL^2|_{L^\prime_i}\eeq on
$\M(\SL_2)$, supported on $L^\prime_i$.  For a variable $s$ we also denote \beq\label{lambdas}\Lambda_{s^2}:= \bigoplus_{j=0}^{2g-2} s^{2j}\Lambda_j.\eeq

 \section{Computation for $i<g-1$}\label{computation} First we start with a definition 
\begin{definition} Let $X$ be a semiprojective $\T$-variety (see \cite{hausel-large}) and $\calF$ an $\T$-equivariant sheaf on it. Then define $$\chi_\T(X;\calF):=\sum_{i,j} (-1)^i \dim(H^i(X;\calF)^{-j}) t^j,$$ where $H^i(X;\calF)^{-j}$ is the $-j$-th weight space of the induced $\T$ action on $H^i(X;\calF)$.
\end{definition}
We expect that for a semi-projective $X$ the equivariant index satisfies $\chi_\T(X;\calF)\in \Z((t))$.
 The following then is one of our main computational tool.
\begin{proposition}\label{smoothformula} For $0\leq i < g-1$  and denoting $v=t^{1\over 2}$ we have 
\begin{multline}\label{main}\chi_\T(\M(\GL_2);\calL_i\otimes\Lambda_{s^2})=\chi_\T(L^\prime_i;\calL^\prime_i\otimes \Lambda_{s^2})= \\  {1\over 2\pi i}\oint_{|z^2-v^2|=\epsilon}z^{4\bi} \frac{ \left((1+ \frac{s}{zv})(1+svz)\right)^{g-1+\bi} \left((1+{sv\over z} )(1+{sz\over v})\right)^{i} }{(1-v^2z^2)^{2\bi+g-1}( 1-v^2/z^2)^{2i-g+1}} \\ \left(4+\frac{  {sv\over z}}{1+{sv \over z}}+\frac{svz {}}{1+svz}+\frac{4v^2z^2}{1-v^2z^2}+{\frac{4v^2}{z^2} \over 1 - \frac{v^2}{z^2}}-\frac{ {s\over zv} }{1+{s\over zv}}-\frac{{sz\over v}}{1+{sz\over v}}\right)^g {dz \over z}\in \Z[[t,s]].
\end{multline}
\end{proposition}
\begin{remark}  In the right hand side of \eqref{main} $v$ is treated as a fixed complex number close to $1$ and $\epsilon$ is sufficiently small positive real number. In particular the contour $|z^2-v^2|=\epsilon$ will have two components around the two square roots $\pm v$ of $t$. Thus the result could also be computed algebraically as the sum of the residues of the differential form when $z$ equals the two square roots of $t$.  
\end{remark}

\begin{proof} As $L^\prime_i$ is the total space of the $\T$-equivariant vector bundle $E_i$, where the $\T$-action has weight two, we have that
\beq\label{integral}\chi_\T(L^\prime_i;\calL^\prime_i\otimes \Lambda_{s^2})=\chi_\T(F^\prime_i;\calL^2\otimes \Lambda_{s^2}\otimes \Sym_{t^2}(E_i^*) )=\int_{F^\prime_i} \ch(\calL^2)\ch(\Lambda_{s^2})\ch(\Sym_{t^2}(E_i^*))\td(T_{F^\prime_i}). \eeq   
Recall \cite[\S 5]{HT1} that we have $$c_1(\Delta_{2i})=2i\otimes \sigma + \sum_{l=1}^{2g} \xi_l\otimes e_l + \eta \otimes 1\in H^2(C_{2i}\times C), $$ where $\sigma$ is generator of $H^2(C)$ and $e_1, \dots, e_{2g}$ are canonical symplectic generators of 
$H^1(C)$. Thus $\eta \in H^2(C_{2i})$ and $\xi_1,
\dots, \xi_{2g} \in H^1(C_{2i})$.  A theorem of Macdonald \cite[(6.3)]{macdonald}
asserts that the cohomology ring $H^*(\rs_m)$ is generated by $\eta$
and the $\xi_j$.  It is convenient to introduce $$\theta_j = \xi_j
\xi_{j+g}$$ and $$\theta = \sum_{j=1}^g \theta_j \in H^2(\rs_{2i}).$$ We note the identity
\beq\label{identity}\left(\sum_{l=1}^{2g} \xi_l\otimes e_l\right)^2=-2\theta\sigma\eeq used below.

By Grothendieck-Riemann-Roch we find that \beq  \ch(E_i)&=&{\pr_{F^\prime_i}}_*\left(\ch(K_C^2 (-\Delta_{2i}))\td(C)\right) \nonumber \\ &=&{\pr_{F^\prime_i}}_*\left(\exp\left((4g-4-2i)\otimes \sigma - \sum_{l=1}^{2g} \xi_l\otimes e_l - \eta \otimes 1\right)(1+(1-g)\sigma)\right)
 \nonumber \\ &=&{\pr_{F^\prime_i}}_*\left(\left((4g-4-2i-\theta)\exp(-\eta)\otimes \sigma - \sum_{l=1}^{2g}\exp(-\eta) \xi_l\otimes e_l + \exp(-\eta) \otimes 1\right)(1+(1-g)\sigma)\right) \nonumber  \\ &=&{\pr_{F^\prime_i}}_*\left(((4g-4-2i-\theta+1-g)\exp(-\eta))\otimes \sigma + \sum_{l=1}^{2g} \exp(-\eta)\xi_l\otimes e_l + \exp(-\eta) \otimes 1\right)  \nonumber \\ &=& (3g-3-2i-\theta)\exp(-\eta) \label{chei} \\ &=&  \nonumber (2g-3-2i)\exp(-\eta)+\sum_{i=1}^g \exp(-\eta-\theta_i)
\eeq 	 	
Thus, introducing $\zeta=\exp(\eta)$, we have \beq\label{symchar }\ch(\Sym_{t^2}E_i^*)=\frac{1}{(1-t^2 \zeta)^{2\bi-1}\prod_{i=1}^g(1-t^2\zeta \exp(\theta_i))}=\frac{1}{(1-t^2 \zeta)^{2\bi+g-1}\exp\left(-\frac{t^2\zeta\theta}{1-t^2\zeta}\right)}\eeq
as \begin{multline*}\prod_{i=1}^g(1-t^2\zeta\exp(\theta_i))=\prod_{i=1}^g(1-t^2\zeta(1+\theta_i))=\prod_{i=1}^g(1-t^2\zeta-t^2\zeta\theta_i)=\\ =(1-t^2\zeta)^g\prod_{i=1}^g\left(1-\frac{t^2\zeta\theta_i}{1-t^2\zeta}\right)=(1-t^2\zeta)^g\prod_{i=1}^g \exp\left(-\frac{t^2\zeta\theta_i}{1-t^2\zeta}\right)=\\ =(1-t^2\zeta)^g\exp\left(-\frac{t^2\zeta\theta}{1-t^2\zeta}\right).\end{multline*}
Similar computation, using the formula for total Chern class \cite[(14.5)]{macdonald}, yields \beq\label{todd}\td(T_{C_{2i}})=\left({\eta\over 1-1/\zeta}\right)^{2i-g+1}\exp\left(
{\theta \over \zeta -1} -{\theta\over \eta} \right).\eeq

To compute the remaining Chern characters in \eqref{integral}. We recall that $$c_1(\calP_{\bi})=\bi\otimes \sigma + \sum_{l=1}^{2g} \tau_l\otimes e_l \in H^2(\Jac_{2\bi}\times C). $$ 
Comparing with \eqref{basic} we see that $$\ch(\calP^2_{2\bi})=\ch(\Delta_i)\ch(K_C^{-1})\ch(M^{-1}_{F^\prime_i}),$$
and the fact that $M_{F^\prime_i}$ is pulled back from $F^\prime_i$ implies that on $F^\prime_i$ we have $$\tau_l=\xi_l/2$$ and 
\beq\label{detlf}\ch(M_{F^\prime_i})=\zeta.\eeq We need now to compute the $\T$-equivariant Chern character of $\bE_i$. From \eqref{univerfi} and the fact that $\T$ acts on the second summand with weight minus one we get $$\ch_\T(\bE_i)=\ch(\calP_{\bi})+t\ch(\calP^{-1}_\bi)\ch(M_{F^\prime_i})=\exp\left(\bi\otimes \sigma + \sum_{l=1}^{2g} \xi_l/2\otimes e_l\right)+t\zeta\exp\left(-\bi\otimes \sigma - \sum_{l=1}^{2g} \xi_l/2\otimes e_l\right).$$ 
It follows from Grothendieck-Riemann-Roch that 
\begin{multline*}\ch_\T(\bV_i)=-\ch_\T(R {\pr_C}_*(\bE_i\stackrel{\bPhi_i}{\to}\bE_iK_C))= {\pr_C}_*\left(\ch_\T(\bE_i)(-1+t^{-1}\ch(K_C))\td(C)\right)\\={\pr_{C}}_*\left(\left(\exp\left(\bi\otimes \sigma + \sum_{l=1}^{2g} \xi_l/2\otimes e_l\right)+t\zeta\exp\left(-\bi\otimes \sigma - \sum_{l=1}^{2g} \xi_l/2\otimes e_l\right)\right) (-1+t^{-1}+(g-1)(1+t^{-1})\sigma))\right) \\ = {\pr_{C}}_*\left(\left(\left(1+(\bi-\theta/4) \otimes \sigma 
  \right) +t\zeta\left(1+(-\bi-\theta/4) \otimes \sigma 
   \right)\right) (-1+t^{-1}+(g-1)(1+t^{-1})\sigma)\right)\\ = \left((\bi-\theta/4)+t\zeta(-\bi-\theta/4)\right)(t^{-1}-1)+(1+t\zeta)(g-1)(t^{-1}+1)\\ = \bi(t^{-1}-1)(1-t\zeta)-\theta/4(t^{-1}-1)(1+t\zeta)+(g-1)(t^{-1}+1)(1+t\zeta)\\ = \bi(t^{-1}-1)(1-t\zeta)-e^{\theta_i/4}(t^{-1}-1)(1+t\zeta)+g(t^{-1}-1)(1+t\zeta)+(g-1)(t^{-1}+1)(1+t\zeta).\end{multline*}
Thus we have \begin{multline}\label{chlambdas}\ch_\T(\Lambda_{s}(\bV_i))=\left(\frac{(1+{s\over t})(1+st\zeta)}{(1+s)(1+s\zeta)}\right)^{\bi}\prod_{i=1}^g\frac{(1+s e^{\theta_i/4})(1+st\zeta e^{\theta_i/4})}{(1+{s\over t} e^{\theta_i/4})(1+s\zeta e^{\theta_i/4})}\left(\frac{(1+{s\over t})(1+s\zeta)}{(1+{s})(1+st\zeta)}\right)^{g} \\ \left((1+{s/ t})(1+s\zeta)(1+s)(1+st\zeta)\right)^{g-1}\\=
\frac{\exp\left(\frac{s{\theta/4}}{1+s}+\frac{st\zeta {\theta/4}}{1+st\zeta}\right)}{\exp\left(\frac{{s\over t}{\theta/4}}{1+{s\over t}}+\frac{s\zeta {\theta/4}}{1+s\zeta}\right)}
 \left((1+{s/t})(1+st\zeta)\right)^{g-1+\bi} \left((1+s)(1+s\zeta)\right)^{i}. \end{multline}

Let $$s^\prime:=\frac{s}{\zeta^{1/2}}.$$ Recalling the extra equivariant twist we endowed $\Lambda_i$ in the paragraphs after \eqref{restrict}, we have \beq\label{extchar} \ch_\T(\Lambda_{s^2})=(\ch_\T(\Lambda_{s^\prime}(\bV_i))+\ch_\T(\Lambda_{-s^\prime}(\bV_i)))/2 . \eeq 

Finally we compute \begin{multline*}\ch_\T(R\pr_{C*}(\bE_i)[1])=\\ ={\pr_C}_*\left(-\left(\exp\left(\bi\otimes \sigma + \sum_{l=1}^{2g} \xi_l/2\otimes e_l\right)+t\zeta\exp\left(-\bi\otimes \sigma - \sum_{l=1}^{2g} \xi_l/2\otimes e_l\right)\right)(1+(1-g)\sigma)\right)=\\ = {\pr_C}_*\left( -\left(1+(\bi-\theta/4) \otimes \sigma \right) -t\zeta\left(1+(-\bi-\theta/4) \otimes \sigma)\right)(1+(1-g)\sigma)\right) =\\ = \bi(t\zeta-1)+\theta/4(1+t\zeta)+(g-1)(t\zeta+1)=\bi(t\zeta-1)+\sum_{i=1}^ge^{\theta_i/4}(1+t\zeta)-(t\zeta+1)
\end{multline*}
thus $$\ch_\T(\det(R\pr_{C*}(\bE_i)[1]))=(t\zeta)^{\bi+g-1}e^{\theta/2}$$ and so by \eqref{determinant} and \eqref{detlf}
\beq\label{detfinal}\ch_\T(\calL|_ {L^\prime_i})=t^\bi\zeta^\bi e^{\theta/2}.\eeq



Recall the following integral formula of Zagier's  of \cite[(7.2)]{thaddeus}. For any power series
$A(x)\in\C[[x]]$ and $B(x)\in\C[[x]]$, 
$$\int_{C_n}A(\eta) \exp(B(\eta)\si) 
= \Res_{x=0} \left( \frac{A(x)(1+x B(x))^g}{x^{n+1}}dx
\right).$$ We also let $$z=e^{x-u\over 2}\in\C[[x,u]],$$ in particular $dx={2dz\over z}.$ Here $u$ is such that $$e^{-u}=t.$$ We also introduce $$v:=e^{-u\over 2}.$$

By recalling that $\pr_{C_{2i}}:F^\prime_i\to C_{2i}$ is a $2^{2g}$ cover we will now compute \begin{multline*} \int_{F^\prime_i} \ch_\T(\calL^\prime_i)\ch_\T(\Lambda_{s^\prime}(\bV_i)) \ch_\T(\Sym_{t^2}E_i^*) \td(T_{F^\prime_i}) =\\ =
\int_{F^\prime_i}(t\zeta)^{2\bi} \frac{\exp\left(\theta+\frac{  s^\prime{\theta/4}}{1+s^\prime}+\frac{s^\prime t\zeta {\theta/4}}{1+s^\prime t\zeta}+\frac{t^2\zeta\theta}{1-t^2\zeta}+{\theta \over \zeta -1}\right)}{\exp\left(\frac{{s^\prime\over t}{\theta/4}}{1+{s^\prime\over t}}+\frac{s^\prime \zeta {\theta/4}}{1+s^\prime \zeta}+{\theta\over \eta}\right)} \frac{ \left((1+{s^\prime\over t})(1+s^\prime t\zeta)\right)^{g-1+\bi} \left((1+s^\prime )(1+s^\prime \zeta)\right)^{i} {\eta}^{2i-g+1}}{(1-t^2 \zeta)^{2\bi+g-1}( 1-1/\zeta)^{2i-g+1}}  =\\ = 2^{2g}\Res_{x=0}z^{4\bi} \frac{ \left((1+ \frac{s}{zv})(1+svz)\right)^{g-1+\bi} \left((1+{sv\over z} )(1+{sz\over v})\right)^{i} {x}^{-g}}{(1-v^2z^2)^{2\bi+g-1}( 1-v^2/z^2)^{2i-g+1}} \\ \left(1+x\left(1+\frac{  {sv\over 4z}}{1+{sv \over z}}+\frac{svz {/4}}{1+svz}+\frac{v^2z^2}{1-v^2z^2}+{\frac{v^2}{z^2} \over 1 - \frac{v^2}{z^2}}-\frac{ {s\over 4zv} }{1+{s\over zv}}-\frac{{sz\over 4v}}{1+{sz\over v}}-{1\over x}\right)\right)^g dx =\\= \Res_{x=0}z^{4\bi} \frac{ \left((1+ \frac{s}{zv})(1+svz)\right)^{g-1+\bi} \left((1+{sv\over z} )(1+{sz\over v})\right)^{i} }{(1-v^2z^2)^{2\bi+g-1}( 1-v^2/z^2)^{2i-g+1}} \\ \left(4+\frac{  {sv\over z}}{1+{sv \over z}}+\frac{svz {}}{1+svz}+\frac{4v^2z^2}{1-v^2z^2}+{\frac{4v^2}{z^2} \over 1 - \frac{v^2}{z^2}}-\frac{ {s\over zv} }{1+{s\over zv}}-\frac{{sz\over v}}{1+{sz\over v}}\right)^g dx =\\= {1\over 2\pi i}\oint_{|z-v|=\epsilon}z^{4\bi} \frac{ \left((1+ \frac{s}{zv})(1+svz)\right)^{g-1+\bi} \left((1+{sv\over z} )(1+{sz\over v})\right)^{i} }{(1-v^2z^2)^{2\bi+g-1}( 1-v^2/z^2)^{2i-g+1}} \\ \left(4+\frac{  {sv\over z}}{1+{sv \over z}}+\frac{svz {}}{1+svz}+\frac{4v^2z^2}{1-v^2z^2}+{\frac{4v^2}{z^2} \over 1 - \frac{v^2}{z^2}}-\frac{ {s\over zv} }{1+{s\over zv}}-\frac{{sz\over v}}{1+{sz\over v}}\right)^g {2dz \over z}.
\end{multline*}

We notice that  $$\int_{F^\prime_i} \ch_\T(\calL_i)\ch_\T(\Lambda_{-s^\prime}(V_i)) \ch_\T(\Sym_{t^2}E_i^*) \td(T_{F^\prime_i})$$ will yield the same result with $z$ replaced by $-z$. 

Thus \eqref{extchar} implies the proposition. 
\end{proof}

\section{The symmetry for $i,j< g-1$}\label{symmetrysmoothi}
First we rewrite the RHS of \eqref{main} in a more manageable form. 

	Introduce the notations \beq \label{bethev} f(z,s,v)=\frac{ z^4 ( 1-v^2/z^2)^2(1+ \frac{s}{zv})(1+svz)}{(1-v^2z^2)^2(1+{sv\over z} )(1+{sz\over v})}\eeq
	and $$h(z,s,v)={z}{{\partial f\over \partial z}}\frac{(1+ \frac{s}{zv})(1+svz)(1+{sv\over z} )(1+{sz\over v})}{(1-{v^2 / z^2})(-v^2z^2+1)}.$$ 
	We can compute \beq \label{hessianz}z\frac{{\partial\over \partial z}f(z,s,v)}{f(z,s,v)}=\left(4+\frac{  {sv\over z}}{1+{sv \over z}}+\frac{svz {}}{1+svz}+\frac{4v^2z^2}{1-v^2z^2}+{\frac{4v^2}{z^2} \over 1 - \frac{v^2}{z^2}}-\frac{ {s\over zv} }{1+{s\over zv}}-\frac{{sz\over v}}{1+{sz\over v}}\right).\eeq
	With this notation, for $i<g-1$ we can rewrite the right-hand side of \eqref{main} as $${1\over 2\pi i}\oint_{|(z^2-v^2)(1+ \frac{s}{zv})(1+svz)|=\epsilon} {h^{g-1}\over f^i} {df \over f},$$ because the integrand has no pole at $z=-s/v$ or $z=-1/vs$. 
	Thus for $0\leq i<g-1$ and $0\leq j \leq g-1$ \eqref{main} gives \beq\label{compact}\chi_\T(\M(\GL_2);\calL_i\otimes\Lambda_j))={1\over (2\pi i)^2}\oint_{|s|=\epsilon}\oint_{|f|=\epsilon} {h^{g-1}\over f^i s^{2j}} {df \over f}{ds \over s}.\eeq

\begin{proposition}\label{symmetryformal} For any integers $i$ and $j$ we have 
	
$$	{1\over (2\pi i)^2}\oint_{|s|=\epsilon}\oint_{|f|=\epsilon} {h^{g-1}\over f^i s^{2j}} {df \over f}{ds \over s}={1\over (2\pi i)^2}\oint_{|s|=\epsilon}\oint_{|f|=\epsilon} {h^{g-1}\over f^j s^{2i}} {df \over f}{ds \over s}.$$
		\end{proposition}
		
	\begin{proof} 		To show that $h$ is "symmetric" in $f$ and $s^2$, we consider $$w={\left( (sz+v)(vz+s)\over (szv+1)(sv+z)\right)}^{1\over 2}.$$
		Then, using \eqref{sothat}, we find $$f(w,f(z,s,v)^{1\over 2},v)=s^2$$ and that $$h(w,f(z,s,v)^{1\over 2},v)=h(z,s,v).$$ One can check these functional equations by noting that $h$ and $f$  depend only on $z^2,s^2$ and $zs$, thus the substitutions  $z=w$ and $s=f^{1\over 2}$ can be done polynomially.  

Thus  consider $\varphi:\C^2\to \C^2$ by $$\varphi(z,s)=(w,f(z,s,v)^{1\over 2})$$ defined on a dense open subset of the source, by avoiding the poles of $f$ and $w$, and also choosing compatible branches of the two square roots, so that \beq\label{sothat}wf^{1\over 2}=\frac{z^2(1-v^2/z^2)(vz+s)}{(1-v^2z^2)(sv+z)}.\eeq Then we see that $$\varphi^*\left({df\over f}\right) = {ds^2\over s^2}={2ds\over s}$$ $$\varphi^*\left({ds\over s}\right) = {d f^{1\over 2} \over f^{1\over 2} }={1\over 2}{df \over f} $$ and so $$\varphi^*\left({{df \wedge ds }\over fs }\right)={{ds \wedge df }\over sf}. $$
		Thus we can compute $$\oint_{|s^2|=|f|=\epsilon} {h^{g-1}\over f^i s^{2j}} {{df \wedge ds }\over fs }=\oint_{|f|=|s^2|=\epsilon} \varphi^{*}\left({h^{g-1}\over f^i s^{2j}} {{df \wedge ds }\over fs } \right)=\oint_{|f|=|s^2|=\epsilon}{h^{g-1}\over s^i f^{2j}} {{ds \wedge df }\over sf }=\oint_{|s^2|=|f|=\epsilon}{h^{g-1}\over s^i f^{2j}} {{df \wedge ds }\over fs },$$
		proving the proposition. 
		
	\end{proof}

\begin{remark} We can observe that $$f(1/z,s,v)=f(z,s,v)^{-1}$$ and that $$\frac{h(1/z,s,v)}{f(1/z,s,v)}=\frac{h(z,s,v)}{f(z,s,v)}.$$
Consequently after the change of variables $z\mapsto 1/z$ we have 
\begin{multline} \oint_{|s^2|=|f|=\epsilon} {h^{g-1}\over f^i s^{2j}} {{df \wedge ds }\over fs } =\oint_{|s^2|=|1/f|=\epsilon} {h^{g-1}\over f^{2g-2-i} s^{2j}} {{-df \wedge ds }\over fs }=\oint_{|s^2|=|f|=\epsilon} {h^{g-1}\over f^{2g-2-i} s^{2j}} {{df \wedge ds }\over fs }, \label{symmetryfz}
\end{multline}
where we used the residue theorem and the fact that the poles of the integrand are located at the zeros of $f$ and $1/f$. 
Using Proposition~\ref{symmetryformal} this yields $$ \oint_{|s^2|=|f|=\epsilon} {h^{g-1}\over f^i s^{2j}} {{df \wedge ds }\over fs } =\oint_{|s^2|=|f|=\epsilon} {h^{g-1}\over f^{i} s^{2(2g-2-j)}} {{df \wedge ds }\over fs }, $$ which is expected in light of the role of $s$ in labelling
the various exterior powers of $\bV$ as in \eqref{lambdas}.
\end{remark}

\begin{remark} \label{vanishing} We can also notice that for $i<0$  \beq\label{zerof}\oint_{|s^2|=|f|=\epsilon} {h^{g-1}\over f^i s^{2j}} {{df \wedge ds }\over fs }=0,\eeq because the integrand has no pole at zeroes of $f$. From Proposition~\ref{symmetryformal} and \eqref{symmetryfz} we get that for $j<0$ and $j>2g-2$  \beq\label{zeros}\oint_{|s^2|=|f|=\epsilon} {h^{g-1}\over f^i s^{2j}} {{df \wedge ds }\over fs }=0.\eeq For $0\leq i < g-1$ this follows from the fact that $\bV$ is a vector bundle of rank $4g-4$ on $\M(\GL_2)^s$. We will also give an interpretation of  this for $i=g-1$ below in Remark~\ref{cohomologicalvb}. 
\end{remark}

Finally by combining the previous Proposition~\ref{symmetryformal} with \eqref{compact} we have now the following result, which is identical to our Theorem~\ref{mainth} from the Introduction.

\begin{corollary}\label{symmetrysmooth} For $0\leq i,j <g-1$ we have $$\chi_\T(\M(\GL_2);\calL_i\otimes\Lambda_j)=\chi_\T(\M(\GL_2);\calL_j\otimes\Lambda_i)$$
\end{corollary}

\section{Computation over the moduli stack}\label{global}

In this section, we will continue the study of the equivariant indices of $\Lambda_j$, but now over the {\it moduli stack} of Higgs bundles. This means that we will also incorporate the contribution of unstable Hitchin pairs. The benefit is that we can now apply the very powerful techniques developed by Teleman and Woodward in \cite{TW}.  

As this approach has the advantage of being easily generalizable to other Lie groups, we will first take $\G$ to be a general connected finite-dimensional reductive group over $\C$, and will only specify it to be $\SL(2,\C)$ or $\GL(2,\C)$ at later stage. We assume $\pi_1(\G)$ has no torsion. Let $\G_0$ be a real form of $\G$, whose fundamental group again is assumed to be free. We consider the moduli space of $\G_0$-Higgs bundles, following \cite{GGM}. The Lie algebra of $\G_0$ is denoted as $\frak{g}_0$ and satisfies $\frak{g}_0\otimes_\R \C \simeq \frak{g}$. It admits a Cartan decomposition
$$
\frak{g}_0=\fh \oplus \fm.
$$
Here $\fh$ is the Lie algebra of the maximal compact subgroup $\H\subset \G_0$, and as the adjoint action of $\H$ preserves the Cartan decomposition, $\fm$ will carry an $\H$-representation. We denote the set of weights as $\frak{R}(\fm)$. Over the curve $C$, a $\G_0$-Higgs bundle is a pair $(E,\varphi)$, where $E$ is a holomorphic principal $\H^\C$-bundle over $C$ and 
$$
\varphi\in H^0(C,E(\fm^\C) \otimes K_C),
$$
where $E(\fm^\C):=E\times_{\H^\C} \fm^\C$ is the $\fm^\C$-bundle associated to $E$.

One can define the moduli space $\M(\G_0)$ of semi-stable $\G_0$-Higgs bundles, over which there is again a $\T$-action acting by scaling the Higgs field, and one can consider the equivariant index of various vector bundles over it. To apply the result of \cite{TW}, we will first work with the moduli stack of $\G_0$-Higgs bundles, defined as follows. Denote $E^*(\fh^\C)$ the vector bundle associated to the universal $\H^\C$-bundle over $C \times \mathrm{Bun}_{\H^\C}$, and $\pi$ the projection $C \times \mathrm{Bun}_{\H^\C} \ra \mathrm{Bun}_{\H^\C}$. Then the moduli stack of $\G_0$-Higgs bundles is 
$$
\frak{M}{(\G_0)}=\mathrm{Spec}\,\mathrm{Sym}\left({\mathbf R}\pi_*(E^*(\fh^\C))[1]
\right).
$$ 

Over $\frak{M}{(\G_0)}$, we say a line bundle $L$ has level $k\in H^4(B\H^\C,\Q)$ if $c_1(L)$ is the image of $k$ under the injective transgression homomorphism $H^4(B\H^\C,\Q) \ra H^2({\mathrm{Bun}_{\H^\C}},\Q)$. $k$ defines a quadratic form on $\fh$, and, following \cite{TW}, we call a line bundle {\it admissible} if $k>-h$ as a quadratic from with $h:=-\frac{1}{2}\mathrm{Tr}_{\fh}$. For a line bundle $L$ of level $k$, we will usually refer to it as $\frak{L}^k$, with slightly abusive use of notation, to make the level explicit. When $H^4(B\H^\C,\Q)=\Z$, $k$ is an integer, and $\frak{L}^{\otimes k}$ will be in the same K-theory class of $L$ if we take $\frak{L}$ to be the determinant line bundle. $\T$ also acts on $\frak{M}{(\G_0)}$ by scaling the Higgs field, and we will consider admissible line bundles that are equivariant, whose equivariant index is what we are after. But before presenting the formula for this index, we first need to introduce some notations. 

Consider the function on the Cartan subalgebra $\ft\subset \fh^\C$
$$
D: \quad \xi \mapsto \frac{k+h}{2}(\xi,\xi)-\Tr_{\frak{h}}\left(\mathrm{Li}_2(te^\xi)\right),
$$
and the map from $T$, the Cartan of $\H^\C$, to $T^\vee$
$$
\chi_t=e^{(k+h)(\cdot,\cdot)}\prod_{\mu\in \frak{R}(\fm)} \left(1-te^\mu\right)^\mu.
$$
We let $F_t$ be the preimage of $e^{2\pi i \rho}$ in the regular part of $T$. It is obvious that $F_t$ is invariant under the action of the Weyl group $W_\H$ of $\H$. Combined with the Killing form, $D$ defines an endomorphism on $\ft$ which we denote as $H^\dagger_t$, and finally we define
$$
\theta_t:=\frac{\prod_{\alpha}(1-e^{\alpha})\prod_{\mu\in \frak{R}(\fm)}(1-te^\mu)}{\det H^\dagger_t},
$$
where $\alpha$ in the first product in the numerator runs over all roots of $\fh$. Then we have the following theorem.

\begin{theorem}[Equivariant Verlinde formula for real groups]
The equivariant index of $\frak{L}^k$ over $\frak{M}{(\G_0)}$ is given by
\begin{equation}\label{IndForm}
\chi_{\T}(\frak{M}(\G_0),\frak{L}^k)=\sum_{f\in F_{t}/W_\H}\theta_t(f)^{1-g}.
\end{equation}
\end{theorem} 

\begin{proof}
From the map $\frak{M}(\G_0)\ra \mathrm{Bun}_{\H^\C}$, we have
$$
\chi_{\T}(\frak{M}(\G_0),\frak{L}^k)=\chi\left(\mathrm{Bun}_{\H^\C},\frak{L}^k\otimes \Lambda_{-t}{R}\pi_*(E^*(\fh^\C))\right).
$$
Here $\chi$ on the right-hand side denotes the usual (non-equivariant) index of a K-theory class over $\mathrm{Bun}_{\H^\C}$. Also, we have abusively used $\frak{L}^k$ to denote its restriction to $\mathrm{Bun}_{\H^\C}$, and used the identity $\mathrm{Sym}_t\left(V[1]\right)=\Lambda_{-t} V$ for arbitrary class K-theory $V$. Then the right-hand side of the equation can be handled by Theorem 2.15 of \cite{TW}. The computation is a straightforward modification of that in the proof of Theorem 7 in \cite{AGP}, with the adjoint representation now replaced with $\fm$.
\end{proof}

\begin{remark} In fact, the theorem holds for more general non-compact subgroups $\G_0\subset \G$ that are not real forms of $\G$. For example, it holds when $\G_0=\G$ is the entire group, for which the index formula just becomes the equivariant Verlinde formula for $\G$ studied in \cite{GP,GPYY,AGP,HL}. 

\end{remark}

\begin{example}[$\G_0=\SL(2,\R)$] In this case $\H=\U(1)$, and $\fm^\C$ is the two-dimensional representation that decomposes as $\C_{+2}\oplus \C_{-2}$, where the subscripts denote the weights of the $\U(1)$-action. We parametrize $\H$ by an angle $\vartheta$ in $[0,2\pi)$. Then $F_t$ is the set of solutions to the equation
$$
e^{2k i\vartheta}\left(\frac{1-te^{-2i\vartheta}}{1-te^{2i\vartheta}}\right)^2=1,
$$
and
$$
\theta_t=\frac{(1-te^{2i\vartheta})(1-te^{-2i\vartheta})}{\det H^\dagger_t},
$$
where $H^\dagger_t$ is the Hessian of 
$$
D: \quad \vartheta \mapsto - k\vartheta^2-\left(\mathrm{Li}_2(te^{2i\vartheta})+\mathrm{Li}_2(te^{-2i\vartheta})\right).
$$

\end{example}

\begin{example}[$\G_0=\U(1,1)$]

This is very similar to the previous case. We now have $\H=\U(1)\times \U(1)$ and $\fm^\C=\C_{1,-1}\oplus\C_{-1,1}$, where subscripts again denote weights under $\U(1)\times \U(1)$. We can identify $\H$ with diagonal $2\times 2$ matrices $\mathrm{diag}\{e^{i\vartheta_1},e^{i\vt_2}\}$, and it is convenient to consider the combination $\vt=(\vt_1-\vt_2)/2$, $\vt'=(\vt_1+\vt_2)/2$, both in $[0,2\pi)$. The level is given by two integers $k$ and $k'$, corresponding to respectively the non-abelian and abelian part of $\U(1,1)$. Now, $F_t$ is given by the solutions to the following set of equations
\begin{align*}
e^{2k i\vartheta}\left(\frac{1-te^{-2i\vartheta}}{1-te^{2i\vartheta}}\right)^2&=1,\\
e^{2k'i\vt'}&=1,
\end{align*}
and
$$
\theta_t=\frac{(1-te^{2i\vt})(1-te^{-2i2\vt})}{\det H^\dagger_t},
$$
where $H^\dagger_t$ is the Hessian of
$$
D: \quad \vartheta \mapsto -k\vartheta^2-k'\vt'^2 -\left(\mathrm{Li}_2(te^{2i\vt})+\mathrm{Li}_2(te^{-2i\vartheta})\right).
$$
It is easy to related the $\U(1,1)$ case to previous case of $\G_0=\SL(2,\R)$ as the equation for $\vt$ are the same, while the new variable $\vt'$ takes value in $\{0,\frac{\pi}{k},\frac{2\pi}{k},\ldots,\frac{(2k-1)\pi}{k}\}$. Using 
$$
\det H^{\dagger,\U(1,1)}_t=\det H^{\dagger,\SL(2,\R)}_t\cdot k',
$$
we have
$$
\chi_{\T}\left(\frak{M}_{\U(1,1)},\frak{L}^k\right)=k'^g\cdot\chi_{\T}\left(\frak{M}_{\SL(2,\R)},\frak{L}^k\right).
$$
\end{example}

In our study of mirror symmetry between moduli spaces of Higgs bundles, it is also useful to consider indices of exterior powers of $\bV$. On the moduli space, $\bV$ can only be defined locally, but it exists globally on the stack $\frak{M}(\SL(2,\C))$ as a hypercohomology complex $\mathbf{R}\pr_{C*}(\bE\stackrel{\bPhi}{\to}\bE\otimes K_C)[1]$, which we will continue call $\bV$. 

There are several moduli stacks we would like to consider, with natural morphisms between them:
$$
\begin{matrix}
\mathrm{Bun}_{\C^\times} & \rightarrow & \mathrm{Bun}_{\SL(2,\C)} \\
 \mathrel{\rotatebox[origin=c]{-90}{$\hookrightarrow$}} & &  \mathrel{\rotatebox[origin=c]{-90}{$\hookrightarrow$}} \\
\frak{M}(\SU(1,1)) & \rightarrow & \frak{M}(\SL(2,\C)) .
\end{matrix}
$$
The vertical maps are injective, while the horizontal ones are not. The image of the bottom map is a substack $\frak{M}'(\SU(1,1))=\frak{M}(\SU(1,1))/W$ where $W=\Z/2\Z$ is the Weyl group of $\SL(2,\C)$. The complex $\bV$ can be restricted to $\frak{M}'(\SU(1,1))$ as well as be pulled back to every corner of the above diagram. To avoid clutter, we will continue denoting the resulting sheaves on $\frak{M}'(\SU(1,1))$ and $\frak{M}(\SU(1,1))$ as $\bV$, while we will denote the pullback to $\mathrm{Bun}_{\C^\times}$ and $\mathrm{Bun}_{\SL(2,\C)}$ in a different font as $\frak{V}$. On the other hand, we will use $\frak{L}^k$ for the $k$-th tensor power of the determinant line bundle that also exists on all the five stacks mentioned above. 

To study the index of $\frak{V}$ and its exterior powers, it is usual to have the following proposition.
\begin{proposition} In the rational equivariant K-theory of $\mathrm{Bun}_{\SL(2,\C)}$, we have 
$$
\mathfrak{V}=(v^{-1}+v)(g-1)\bE_{x}+(v^{-1}-v)\bE_{C},
$$
where $\bE_{x}$ and $\bE_{C}$ are obtained from the rank $2$ universal bundle $\bE$ over $\mathrm{Bun}_{\SL(2,\C)} \times C$ by, respectively, restricting to a point $x\in C$ and slanting with the fundamental class of $C$. 
\end{proposition}
\begin{proof} As $\frak{V}=v^{-1}{R}\pi_*(\bE\otimes K_C)-v{R}\pi_*(\bE)$, the proposition directly follows from two identities in the K-theory of $\mathrm{Bun}_{\SL(2,\C)}$,
$$
{R}\pi_*(\bE)=(1-g)\bE_{x}+\bE_{C},
$$
and
$$
{R}\pi_*(\bE\otimes K_C)=(g-1)\bE_{x}+\bE_{C}.
$$
\end{proof}

Now we consider the index of $\Lambda_s\bV$ on the moduli stack of $\SU(1,1)$-Higgs bundles $\frak{M}(\SU(1,1))$. As the maximal compact subgroup of $\SU(1,1)$ is $\U(1)$, we use the morphism 
$$
p: \frak{M}(\SU(1,1))\quad\ra\quad\mathrm{Bun}_{\C^\times}=\mathrm{Pic}(C)\times B\C^\times
$$
to reduce the index computation over $\frak{M}(\SU(1,1))$ to an index computation over $\mathrm{Bun}_{\C^\times}$. To give the index formula explicitly, we parametrize the complexified maximal compact subgroup $\H^\C=\C^\times$ by $z$, and denote the set of solutions of
\beq \label{thetabethe}
z^{2k}\left(\frac{1-v^2z^{-2}}{1-v^2z^2}\right)^2\left(\frac{1+sv^{-1}z^{-1}}{1+sv^{-1}z}\right)\left(\frac{1+svz}{1+svz^{-1}}\right)=1,
\eeq
as $F'_{t,s}$. Define
\beq \label{littletheta}
\theta'_{t,s}:=\frac{1}{\det H'^\dagger_{t,s}}\cdot \frac{(1-v^2z^2)(1-v^2z^{-2})}{(1+sv^{-1}z)(1+sv^{-1}z^{-1})(1+svz)(1+svz^{-1})},
\eeq
where $\det H'^\dagger_{t,s}$ is given by
\beq\label{Hessian}
\left[\left(2k + \frac{4v^2z^2}{1-v^2z^2}+\frac{4v^2z^{-2}}{1-v^2z^{-2}}-\frac{sv^{-1}z}{1+sv^{-1}z}-\frac{sv^{-1}z^{-1}}{1+sv^{-1}z^{-1}}+\frac{svz}{1+svz}+\frac{svz^{-1}}{1+svz^{-1}}\right) \right].
\eeq

Then we have
\begin{theorem}\label{IndSU11}
For $\frak{L}^k$ a line bundle on $\mathfrak{M}(\SU(1,1))$ with level $k\geq 0$, the $\T$-equivariant index of $\frak{L}^k\otimes \Lambda_s\bV$ is given by
\beq\label{IndForm3} 
\chi_\T\left(\mathfrak{M}(\SU(1,1)),\frak{L}^k\otimes \Lambda_s \bV \right)=\sum_{f\in F'_{t,s}}\theta'_{t,s}(f)^{1-g}.
\eeq
\end{theorem}

\begin{proof}

When $s=0$, the theorem becomes the Verlinde formula for $\SU(1,1)$-Higgs bundles discussed earlier. We now highlight the effect of turning on $s$. 

The first step is to rewrite $\Lambda_s \mathfrak{V}$ on $\mathrm{Bun}_{\C^\times}$ into a form that Theorem 2.15 of \cite{TW} can handle. To achieve this, we will start with rewriting $\Lambda_s \mathfrak{V}$ over $\mathrm{Bun}_{\SL(2,\C)}$. Recall first that from the universal $\SL(2,\C)$-bundle over $\mathrm{Bun}_{\SL(2,\C)}\times C$ and a $\SL(2,\C)$ representation $R$, one can obtain a vector bundle, denoted as $E(R)$. We obviously have $\bE=E(R_F)$, where $R_F$ is the two-dimensional fundamental representation of $\SL(2,\C)$. One can obtain classes $E_x(R)$ and $E_C(R)$ in the K-theory of $\mathrm{Bun}_{\SL(2,\C)}$ similar to the construction of $\bE_x$ and $\bE_C$. Using the identity
$$
\Lambda_{s}=\exp\left[-\sum_{p>0}\frac{(-s)^p\psi^p}{p}\right]
$$
with $\psi^p$ denoting the $p$-th Adams operation and the relation $\psi^p E_C(R)=E_C(\psi^p R)/p$, one can rewrite
\beq\label{LambdaS}
\Lambda_s \mathfrak{V}=
E_x(\Lambda_s((v+v^{-1})R_F))^{\otimes (g-1)}\otimes\exp\left[-\sum_{p=1}^\infty \frac{(-s)^p E_C(\psi^p ((v^{-1}-v)R_F))}{p^2}\right].
\eeq

Similarly, on $\mathrm{Bun}_{\C^\times}$, we denote $E(R_n)$ the line bundle associated to the universal line bundle via the weight-$n$ representation over $\mathrm{Bun}_{\C^\times}\times C$. As before, we can construct K-theory class $E_x(R_n)$ and $E_C(R_n)$ over $\mathrm{Bun}_{\C^\times}$. From \eqref{LambdaS}, one can deduce the following identity in the K-theory of $\mathrm{Bun}_{\C^\times}$,
\beq\label{pLambdaS}
\Lambda_s\mathfrak{V}=E_x(\Lambda_s((v+v^{-1})(R_1+R_{-1})))^{\otimes (g-1)}\otimes\exp\left[-\sum_{p=1}^\infty \frac{\left((-s/v)^p-(sv)^p\right) E_C(R_p+R_{-p})}{p^2}\right].
\eeq
Since 
$$
\mathfrak{M}(\SU(1,1))=\mathrm{Spec}\mathrm{Sym}({R}\pi_*E(R_2+R_{-2})[1]),
$$
we have the equality
$$
\chi_\T\left(\mathfrak{M}(\SU(1,1)),\frak{L}^k\otimes \Lambda_s \bV \right)=\chi\left(\mathrm{Bun}_{\C^\times},\frak{L}^k\otimes \Lambda_s \mathfrak{V}\otimes \Lambda_{-t}{R}\pi_*E(R_2+R_{-2}) \right),$$
 and the theorem follows from a direct computation using Theorem 2.15 of \cite{TW} applied to $\mathrm{Bun}_{\C^\times}$. With the parametrization of the complexified maximal compact subgroup of $\SU(1,1)$ by $z$, the character of $R_n$ is given by $z^n$. Then Theorem 2.15 of \cite{TW} tells us that the first factor $E_x(\Lambda_s((v+v^{-1})(R_1+R_{-1}))^{\otimes (g-1)}$ in \eqref{pLambdaS} will result in an additional $s-$dependent factor 
$$\mathrm{Tr}_{R_s} (z)$$ 
in $\theta'^{1-g}_{s,t}$, with $R_s$ being the virtual representation $\left(\Lambda_s\left((v+v^{-1})(R_1+R_{-1})\right)\right)^{\otimes (g-1)}$. It is easy to see that the character of $R_s$ is given by 
$$\left[(1+svz)(1+svz^{-1})(1+sv^{-1}z^{-1})(1+sv^{-1}z)\right]^{g-1}.$$ 

On the other hand, the second factor in \eqref{pLambdaS} will modify both $\theta'_{s=0,t}$ and $F'_{s=0,t}$. The modification to the former is only through $\det H'^\dagger_{t,s}$, which is now given by the Hessian of the function 
$$
D'_s=D'_{s=0}+\sum_{p=1}^\infty\frac{\left((-sv)^p-(-s/v)^p\right)\mathrm{Tr}_{\psi^p (R_1+R_{-1})}(z)}{p^2}.
$$
Since $\mathrm{Tr}_{\psi^p (R_1+R_{-1})}(z)=z^p+z^{-p}$, the summation over $p$ in the above can be written as the sum of four terms,
$$
\sum_{p=1}^\infty\frac{\left((-sv)^p-(-s/v)^p\right)\left(z^p+z^{-p}\right)}{p^2}=\mathrm{Li}_2(-svz^{-1})+\mathrm{Li}_2(-svz)-\mathrm{Li}_2(-sv^{-1}z)-\mathrm{Li}_2(-sv^{-1}z^{-1}).
$$
Computing the Hessian of $D'_s$ gives the desired $\det H'^\dagger_{t,s}$.

Now we determine the set $F'_{s,t}$. When $s=0$, it is given by the set of solutions to the ``Bethe ansatz equations'',
$$
z^{2k}\left(\frac{1-v^2z^{-2}}{1-v^2z^2}\right)^2=1.
$$
When $s$ is turned on, from Theorem 2.15 of \cite{TW}, a new factor will be included on the left-hand side of the above equation, given by
$$
\exp\left[-\sum_{p=1}^\infty\frac{\left((-sv)^p-(-s/v)^p\right)(z^p-z^{-p})}{p}\right]=\left(\frac{1+sv^{-1}z^{-1}}{1+sv^{-1}z}\right)\left(\frac{1+svz}{1+svz^{-1}}\right).
$$

Taking into account all three modifications proves the theorem.

\end{proof}

In previous sections, we studied the index over the smooth part of the ``real slice'' $\frak{M}'(\SU(1,1))\subset \frak{M}(\SL(2,\C))$, for which we have the following. 

\begin{theorem} \label{globalindex} The equivariant index of $\frak{L}^k\otimes \Lambda_s\bV$ on $\mathfrak{M}'(\SU(1,1))$
\beq\label{IndForm3} 
\chi_\T\left(\mathfrak{M}'(\SU(1,1)),\frak{L}^k\otimes \Lambda_s \mathfrak{V} \right)=\frac{1}{2}\sum_{f\in F'_{t,s}}\theta'_{t,s}(f)^{1-g},
\eeq
equals to one-half of the equivariant index $\chi_\T\left(\mathfrak{M}(\SU(1,1)),\frak{L}^k\otimes \Lambda_s \bV \right)$ on $\mathfrak{M}(\SU(1,1))$.
\end{theorem}

\begin{proof}
This theorem follows from the fact that $\mathfrak{M}'(\SU(1,1))=\mathfrak{M}(\SU(1,1))/W$ and $\frak{L}^k\otimes \Lambda_s\bV$ is invariant under the action of the Weyl group $W=\Z/2\Z$. 

\end{proof}

\section{Symmetry for $i=g-1$} \label{symmetryg-1}
To compare the formula \eqref{IndForm3} for $k=2$ with the formula \eqref{compact} we rewrite it in the same residue form. First we note  the LHS of \eqref{thetabethe} for $k=2$ equals $f(z,s,v)$ from  \eqref{bethev}.
We also see that \eqref{Hessian}  agrees with  $$z\frac{{\partial\over \partial z} f(z,s,v)}{f(z,s,v)}$$ which is computed in \eqref{hessianz}. Finally we see from \eqref{littletheta} that $$\theta^\prime_{t,s} = \frac{f(z,s,v)}{h(z,s,v)}.$$
By recalling that $F^\prime_{t,s}$ denotes the set of  solutions of $f=1$, which are all simple zeroes, we can  compute \begin{multline*}\sum_{\varphi\in F'_{t,s}}\theta'_{t,s}(\varphi)^{1-g}=\sum_{a\in F'_{t,s}}\left(\frac{h(a,s,v)}{f(a,s,v)}\right)^{g-1}= \frac{1}{2\pi i} \oint_{|1-f|=\epsilon} \left(\frac{h}{f}\right)^{g-1} \frac{1}{1-f} {-df\over f}=\\= \frac{1}{2\pi i}  \oint_{|f|=\epsilon} \left(\frac{h}{f}\right)^{g-1} \frac{1}{1-f} {df\over f} +\frac{1}{2\pi i}   \oint_{|z^2v^2-1|=\epsilon} \left(\frac{h}{f}\right)^{g-1} \frac{1}{1-f} {df\over f}=\\= \frac{1}{2\pi i}  \oint_{|f|=\epsilon} \sum^\infty _{j=0} \left(\frac{h}{f}\right)^{g-1} f^{j}{df\over f} -\frac{1}{2\pi i} \oint_{|z^2v^2-1|=\epsilon} \sum^\infty _{j=1} \left(\frac{h}{f}\right)^{g-1} f^{-j}{df\over f} =\\= \frac{1}{2\pi i}\oint_{|f|=\epsilon} \sum^\infty _{j=0} \left(\frac{h}{f}\right)^{g-1} f^{j}{df\over f} +\frac{1}{2\pi i} \oint_{|f|=\epsilon} \sum^\infty _{j=1} \left(\frac{h}{f}\right)^{g-1} f^{-j}{df\over f} =\\ =  \frac{1}{2\pi i}\oint_{|f|=\epsilon} \left(\frac{h}{f}\right)^{g-1}{df\over f} + \sum_{i=0}^{g-2} \chi_\T(\M(\GL_2); \calL_i\otimes \Lambda_{s^2}) +\frac{1}{2\pi i}  \oint_{|f|=\epsilon} \sum^\infty _{j=1} \left(\frac{h}{f}\right)^{g-1} f^{j}{df\over f} =\\ =  \frac{1}{2\pi i}  \oint_{|f|=\epsilon} \left(\frac{h}{f}\right)^{g-1}{df\over f} + 2 \sum_{i=0}^{g-2} \chi_\T(\M(\GL_2); \calL_i\otimes \Lambda_{s^2}). 
\end{multline*}

In the third equation we used the residue theorem and the fact that the poles of the integrand are the zeroes of $1-f$, $f$ and $z^2v^2-1$. In the fourth that the power series expansions converge on the respective contours. In the fifth equation we used the residue theorem again, noting
that the only poles of the integrand are the zeroes  of $f$ and $z^2v^2-1$. In the sixth \eqref{compact}, \eqref{zerof} and \eqref{symmetryfz}. Finally in the seventh equation we used \eqref{compact} and \eqref{zerof} again. 

Thus from \eqref{globalindex} \beq \label{indexstack}
\chi_\T\left(\mathfrak{M}'(\SU(1,1)),\frak{L}^2\otimes \Lambda_s \mathfrak{V} \right)= {1\over 2}\frac{1}{2\pi i}  \oint_{|f|=\epsilon} \left(\frac{h}{f}\right)^{g-1}{df\over f} + \sum^{g-1} _{j=1} \frac{1}{2\pi i}\oint_{|f|=\epsilon} \left(\frac{h}{f}\right)^{g-1} f^{j}{df\over f}
\eeq

	\begin{remark} \label{cohomologicalvb} From the observation in Remark~\ref{vanishing} we see that   even though $\bV$ and consequently $\Lambda_s$ are no longer vector bundles on $\mathfrak{M}'(\SU(1,1))$, cohomologically $\bV$ behaves like a rank $4g-4$ vector bundle on $\mathfrak{M}'(\SU(1,1))$. Interestingly, the same observation about the virtual Dirac bundle $\bV$ was the crucial point in the proof of the vanishing of the intersection form on the rank $2$ odd degree Higgs moduli space in \cite{hausel-thesis,hausel-vanishing}.
\end{remark}

From \eqref{indexstack} we also get 
\bes 2 \chi_\T\left(\mathfrak{M}'(\SU(1,1)),\frak{L}^2\otimes \Lambda_s \mathfrak{V} \right)- 2 \sum_{i=0}^{g-2} \chi_\T(\M(\GL_2); \calL_i\otimes \Lambda_{s^2}) = \frac{1}{2\pi i}  \oint_{|f|=\epsilon} \left(\frac{h}{f}\right)^{g-1}{df\over f} 
\ees

We can now combine Theorem~\ref{globalindex} and Proposition~\ref{symmetryformal} with the computation above and deduce the following

\begin{corollary} \label{mainc} For $0\leq j < g-1$ we have 
		\bes 2 \chi_\T\left(\mathfrak{M}'(\SU(1,1)),\frak{L}^2\otimes \Lambda^{2j} \mathfrak{V} \right)- 2 \sum_{i=0}^{g-2} \chi_\T(\M(\GL_2); \calL_i\otimes \Lambda_j)  =\chi_\T(\M(\GL_2);\calL_j\otimes\Lambda_{g-1})\ees
\end{corollary}

\begin{remark} For a completely symmetrical picture we would like to see the left hand side agree with $\chi_\T(\M(\GL_2); \calL_{g-1}\otimes \Lambda_j)$ with the special convention for ${\cal L}_{g-1}$ in \eqref{lg-1}. It seems however that the non-semi-stable locus in $\mathfrak{M}'(\SU(1,1))$ will contribute non-trivially to the equivariant index. We expect that a modification of the extension of the Dirac complex $\Lambda_s\bV$ to $\mathfrak{M}(\SL_2)$ will solve this problem. 
	\end{remark}

\section{Reflection of mirror symmetry}
\label{reflect}

If $\calF_1$ and $\calF_2$ are $\T$-equivariant coherent sheaves on a semi-projective (cf. \cite{hausel-large}) $\T$-variety $X$, then we can define (cf. \cite[(2.23),(3.4)]{gukov}) the {\em equivariant Euler form} as \begin{multline*}\chi_\T(X;\calF_1,\calF_2)=\sum_{k,l} \dim(H^k({ R} {\mathcal Hom}(\calF_1,\calF_2))^l) (-1)^k t^{-l}=\\ =\sum_{k,l} \dim\left(\Hom_{{\mathbf D}_{coh}(X)}(\calF_1,\calF_2[k])^l\right) (-1)^k t^{-l}=\sum_{k,l} \dim(\Ext^k(X;\calF_1,\calF_2)^l) (-1)^k t^{-l} \in \C((t)),\end{multline*}
which we can compute by Hirzebruch-Riemann-Roch and localization: $$\chi_\T(X;\calF_1,\calF_2)=\int_{X}\ch_\T(\calF_1)^*\ch_\T(\calF_2)\td_\T(X)=\int_{X^\T} \ch_\T(\calF_1)^*|_{X^\T}\ch_\T(\calF_2|_{X^\T}) \ch_\T(\Sym(N^*_{X^\T}))  \td({X^\T}),$$
where $N_{X^\T}$ is the normal bundle of $X^\T$ in $X$ and for $$a=a_0+a_2+\dots +a_{2\dim(X)}\in H^{2*}(X).$$ We define the class $$a^*=\sum (-1)^i a_{2i}\in H^{2*}(X).$$

\begin{corollary}\label{mirror} When $0\leq i,j<g-1$ $$\chi_\T(\M(\GL_2);\calL_i,\Lambda_j)=-t^{g-1}\chi_\T(\M(\GL_2);\Lambda_i,\calL_j).$$
\end{corollary}

\begin{proof} We expect that \beq\label{selfli}\calL_i^\vee \cong \calL_i[-4g+3]\eeq and \beq\label{selflai}{\Lambda_i}^\vee \cong \Lambda_i[-2g]\eeq self-dualities in the derived category $D_{coh}(\M(\GL_2))$. Indeed, \eqref{selflai} holds on $\M(\GL_2)^s$ because there it is a pushforward of a vector bundle from $\M(\SL_2)^s$ of codimension $2g$ with trivial normal bundle, and as \cite[\S 7, Remarks 1, 2]{hitchin} argues the vector bundle $\Lambda_i|_{\M(\SL_2)^s}$ is self-dual, due to the quaternionic structure on it. On the other hand
 \eqref{selfli} holds for $i<g-1$, because of \cite[Corollary 3.40]{huybrechts}, using that $\M(\GL_2)^s$ is symplectic  
and by the following 
\begin{lemma} For $0\leq i < g-1$ we have $$K_{L_i}=\calL^4|_{L_i}=\calL^2_i|_{L_i}.$$
\end{lemma}
\begin{proof} This follows from $c_1(K_{L^\prime_i}|_{F_i^\prime})=c_1(\calL^4|_{F^\prime_i}).$ Starting with \cite[(14.10)]{macdonald} and \eqref{chei}, we compute $$c_1(K_{L^\prime_i}|_{F_i^\prime})=c_1(K_{F^\prime_i})c_1( E^*_i)=(g-2i-1)\eta+\theta+(3g-3-2i)\eta +\theta=(4g-4-4i)\eta+2\theta=4c_1(\calL|_{F^\prime_i}).$$
from \eqref{detfinal}
\end{proof}

To also trace how the $\T$-action intertwines these isomorphisms we  compute using \cite[Lemma 5.4.9]{chriss-ginzburg}

\begin{multline*}\ch_\T(\calL_i)^*|_{L_i}=\ch_\T\left(\sum_j(-1)^j\Lambda_jN^*_{L_i}\right)^*\ch_\T(\calL^{-2}|_{L_i})=\ch_\T\left(\sum_j(-1)^j\Lambda_j N_{L_i}\right)\ch_\T(\calL^{-2}|_{L_i})=\\ =\frac{(-1)^{4g-3}e^{c^\T_1(N_{L_i})}}{\ch_\T(\Sym N_{L_i}^*)}\ch_\T(\calL^{-2}|_{L_i})\end{multline*}
Consequently, \begin{multline}  \ch_\T(\calL_i)^*|_{F^\prime_i}=\frac{(-1)^{4g-3}e^{c^\T_1(N_{L_i}|_{F^\prime_i})}}{\ch_\T(\Sym N_{L_i}^*)}\ch_\T(\calL^{-2}|_{F^\prime_i})=\frac{(-1)^{4g-3}e^{c^\T_1(E^*_i)+c^\T_1(T^*_{F_i})+c^\T_1(\calL^{-2}|_{F^\prime_i})}}{\ch_\T(\Sym N_{L_i}^*)}=\\  =\frac{(-1)^{4g-3}e^{(g-1+2\bi)(\eta-u)+\theta+(g-1-2i)\eta+\theta+(g+2i)u-2\bi(\eta-u)-\theta}}{\ch_\T(\Sym N_{L_i}^*)}=\frac{-\ch_\T(\calL^2|_{F^\prime_i})t^{-1}}{\ch_\T(\Sym N_{L_i}^*)} \label{lidual} =-\ch_\T(\calL_i|_{F^\prime_i})t^{-1}.\end{multline}
Here we used $$c_1^\T(E_i^*)=(g-1+2\bi)(\eta-u)+\theta$$ from \eqref{chei} and the fact that $\rank(E_i^*)=g-1+2\bi$ and the $\T$-action has weight $-1$ on $E_i^*$. We also used $$c_1^\T(T^*_{F_i})=(g-1-2i)\eta+\theta+(g+2i)u$$ from \cite[(14.10)]{macdonald} and $\rank(T^*_{F_i})=g+2i$ and the $\T$-action has weight one on $T^*_{F_i}$. Finally, $$c^\T_1(\calL^{-2}|_{F^\prime_i})=-2\bi(\eta-u)-\theta$$ from \eqref{detfinal} .
Similarly, $$\ch(\Lambda_j)^*|_{F^\prime_i}=\frac{(-1)^{2g}t^{-g}\ch_\T(\Lambda^{j*}|_{F_i^\prime})}{\ch_\T(\Sym N^{F_i}_{F_i^\prime})}=\frac{t^{-g}\ch_\T(\Lambda_j|_{F_i^\prime})}{\ch_\T(\Sym N^{F_i}_{F_i^\prime})},$$
because from \eqref{chlambdas} and \eqref{extchar} we can deduce $$\ch_\T(\Lambda^{j*}|_{F^\prime_i})=\ch_\T(\Lambda_j|_{F^\prime_i}).$$

Thus we can compute from \eqref{lidual} \begin{multline*}\chi_\T(\M(\GL_2);\calL_i,\Lambda_j)=\int_{F^\prime_i} \ch_\T(\calL_i)^*|_{F^\prime_i}\ch_\T(\Lambda_j|_{F^\prime_i})\ch_\T(\Sym N_{F_i^\prime}^*)\td(F^\prime_i)=\\ =\int_{F^\prime_i}-t^{-1}\ch(\calL^2|_{F^\prime_i})\ch_\T(\Lambda_j|_{F^\prime_i})\ch_\T(\Sym E^*_i)\td(F^\prime_i)=-t^{-1}\chi_\T(\M(\GL_2);\calL_i\otimes \Lambda_j).\end{multline*}

On the other side we have

 \begin{multline*}\chi_\T(\M(\GL_2);\Lambda_i,\calL_j)= \int_{F^\prime_j} \ch_\T(\Lambda_i)^*|_{F^\prime_i}\ch_\T(\calL_j|_{F^\prime_j})\ch_\T(\Sym N_{F_i^\prime}^*)\td(F^\prime_i)=\\ = \int_{F^\prime_i} t^{-g} \ch_\T(\Lambda_i|_{F^\prime_j})\ch_\T(\Sym N_{F_i^\prime}^*)\td(F^\prime_i)= t^{-g}\chi_\T(\M(\GL_2);\Lambda_i\otimes \calL_j),\end{multline*}
and the result follows.
\end{proof}
\begin{remark} Based on the Fourier-Mukai transform on the generic fibers of the Hitchin map, we expect that (the semi-classical limit of the) mirror of $\calL_i$ will be the shifted coherent sheaf $\Lambda_i[-g]$ and that of $\Lambda_j$ the shifted coherent sheaf $\calL_j[-3g+3]$. This explains the sign in Corollary~\ref{mirror}.
\end{remark}
\begin{remark} The extra $t^{g-1}$ factor in Corollary~\ref{mirror} seems to be necessary, and the reason for its appearance could be the fact
that the canonical bundle $K_{\M(\SL_2)}$ is trivial with a weight $3g-3$  equivariant structure, therefore has no square root for $g$ even.  
\end{remark}


\begin{thebibliography}{999}
	
\bibitem[AGP]{AGP}{\sc  Andersen, J.~E., Gukov, S. {\rm and} Pei, D.}: The Verlinde formula for Higgs bundles.
  \href{http://arxiv.org/abs/1608.01761}{arXiv:1608.01761}

\bibitem[BS]{baraglia-schaposnik}{\sc Baraglia, D. {\rm and} Schaposnik, L.}:
Real structures on moduli spaces of Higgs bundles. 
{\em Adv. Theor. Math. Phys.} {\bf 20} (2016), no. 3, 525--551.
	
\bibitem[BGG]{bradlow-etal}{\sc Bradlow, S.B., García-Prada, O. \rm{and} Gothen, P.B.}: Surface group representations 
and $U(p, q)$-Higgs bundles. {\em J. Differ. Geom.} {\bf 64}, 111--170 (2003)

\bibitem[CG]{chriss-ginzburg} {\sc Chriss, N. {\rm and} Ginzburg, V.}
{\em Representation theory and complex geometry.}  Birkhäuser Boston, Inc., Boston, MA, 1997.

	\bibitem[DP1]{donagi-pantev}{\sc Donagi, R. {\rm and} Pantev, T.}: Langlands duality for Hitchin systems, {\em Invent. Math.} {\bf 189} 3, 653--735. \href{http://arxiv.org/abs/math/0604617}{arXiv:math/0604617}
	
\bibitem[GGM]{GGM}{\sc {Garcia-Prada}, O., {Gothen}, P.~B. {\rm and} {Riera}, I.~M.~i}: The Hitchin-Kobayashi correspondence, Higgs pairs and surface group representations, \href{https://arxiv.org/abs/0909.4487}{arXiv:0909.4487}
	
\bibitem[GWZ]{groechenig-etal} {\sc Gr\"ocheing, M., Wyss, D. {\rm and} Ziegler, P.}:  Topological mirror symmetry via p-adic integration,   \href{https://arxiv.org/abs/1707.06417}{arXiv:1707.06417}

\bibitem[Gu]{gukov} {\sc  Gukov, S.}:
Quantization via mirror symmetry. {\em
Jpn. J. Math.} {\bf 6} 2011, no. 2, 65--119. \href{https://arxiv.org/abs/1011.2218}{arXiv:1011.2218}
	
	\bibitem[GP]{GP} {\sc Gukov, S. {\rm and} Pei, D.}: Equivariant Verlinde formula from fivebranes and vortices. \href{https://arxiv.org/abs/1501.01310}{arXiv:1501.01310} (2015).
	
	\bibitem[GPYY]{GPYY} {\sc Gukov, S., Pei, D., Yan, W. {\rm and} Ye, K.}: Equivariant Verlinde algebra from superconformal index and Argyres-Seiberg duality. \href{https://arxiv.org/abs/1605.06528}{arXiv:1605.06528} (2016).
	
	\bibitem[GW]{GW} 
  {\sc Gaiotto, D. {\rm and} Witten, E.}: S-Duality of Boundary Conditions In N=4 Super Yang-Mills Theory,
  {\em Adv.\ Theor.\ Math.\ Phys.\ }  {\bf 13}, no. 3, 721 (2009)
	
	
	\bibitem[Ha1]{hausel-thesis}{\sc  Hausel, T.}: Geometry of the moduli space
	of Higgs bundles, PhD thesis, University of Cambridge, 1998, \href{http://arxiv.org/abs/math.AG/0107040}{arXiv:math.AG/0107040} 

\bibitem[Ha2]{hausel-vanishing}
{\sc Hausel., T.}:
\newblock Vanishing of intersection numbers on the moduli space of {H}iggs
bundles.
\newblock {\em Adv. Theor. Math. Phys.}, 2(5):1011--1040, 1998.


	\bibitem[HL]{HL} {\sc {Halpern-Leistner}, D.}: The equivariant Verlinde formula on the moduli of Higgs bundles. \href{https://arxiv.org/abs/1608.01754}{arXiv:1608.01754}.
	

\bibitem[HV]{hausel-large} {\sc Hausel,  T. {\rm and} Fernando, R. V.}:
Cohomology of large semiprojective hyperk\"ahler varieties. 
{\em Ast\'erisque} No. {\bf 370} (2015), 113--156, \href{https://arxiv.org/abs/1309.4914}{arXiv:1309.4914}

	\bibitem[HT1]{HT1} {\sc Hausel, T. {\rm and} Thaddeus, M.}:
	\newblock { Generators for the cohomology ring of the moduli space of rank 2 Higgs bundles }, {\em Proc. London Math.\
		Soc.} {\bf 88} (2004) 632--658,  \href{http://arxiv.org/abs/math/0003093}{\tt arXiv:math.AG/0003093}
	
	\bibitem[HT2]{HT2}  {\sc Hausel, T. {\rm and} Thaddeus, M.}:
	\newblock   Relations in the cohomology ring of the moduli space of rank 2 Higgs bundles,\newblock {\em Journal of the American Mathematical \ Society}, {\bf 16} (2003), 303-329,  \href{http://arxiv.org/abs/math/0003094}{\tt arXiv:math.AG/0003094}
	
	\bibitem[HT]{hausel-thaddeus}
	{\sc Hausel, T. {\rm and} Thaddeus, M.}:
	\newblock Mirror symmetry, {L}anglands duality, and the {H}itchin system.
	\newblock {\em Invent. Math.}, 153(1):197--229, 2003. \href{https://arxiv.org/abs/math/0205236}{arXiv:math/0205236}
	
\bibitem[HSz]{hausel-szenes}{\sc Hausel, T. {\rm and} Szenes, A.}: Equivariant Verlinde algebra for Higgs bundles, {\em in preparation}

		\bibitem[Hi1]{hitchin-self}
	{\sc Hitchin, N.}:
	\newblock The self-duality equations on a {R}iemann surface.
	\newblock {\em Proc. London Math. Soc. (3)}, 55(1):59--126, 1987.
	

\bibitem[Hi2]{hitchin}{\sc Hitchin, N.}: Higgs bundles and characteristic classes, in {
\em Arbeitstagung Bonn 2013; In Memory of Friedrich Hirzebruch}, Progress in Mathematics, Birkh\"auser, 2016, 247--264. \href{https://arxiv.org/abs/1308.4603}{\tt arXiv:1308.4603}

\bibitem[Hu]{huybrechts}{\sc  Huybrechts, D.}:
{ \em Fourier-Mukai transforms in algebraic geometry.}
Oxford Mathematical Monographs. Oxford University Press, Oxford, 2006.

\bibitem[KW]{kapustin-witten}{\sc Kapustin, K. {\rm and} Witten. E.}:
\newblock Electric-magnetic duality and the geometric {L}anglands program. \newblock {\em Commun. Number Theory Phys.}  1  (2007),  no. 1, 1--236. 
\newblock \href{http://arxiv.org/abs/hep-th/0604151}{arXiv:{hep-th/0604151}}, 2006.

\bibitem[Mac]{macdonald} {\sc Macdonald, I.G.}: Symmetric products of an algebraic
curve, {\em Topology} {\bf 1} (1962) 319-343

\bibitem[Qi]{quillen} {\sc Quillen, D.}:
	Determinants of Cauchy-Riemann operators on Riemann surfaces. 
	Funktsional. Anal. i Prilozhen. {\bf 19} (1985), no. 1, 31--34


\bibitem[Sch]{schaposnik}  {\sc Schaposnik, L. P.}: Spectral data for $U(m,m)$-Higgs bundles. {\em Int. Math. Res. Not. IMRN} no. {\bf 11}, 2015, 3486--3498. \href{https://arxiv.org/abs/1307.4419}{\tt arXiv:1307.4419}

\bibitem[Th]{thaddeus} 
{\sc Thaddeus M.}:
Stable pairs, linear systems and the Verlinde formula,
{\em Invent.\ Math.\ }{\bf 117} (1994) 317--353. 

\bibitem[TW]{TW} {\sc Teleman, C. {\rm and} Woodward, C.T.}: The index formula for the moduli of G-bundles on a curve. {\em Annals of mathematics} (2009) 495-527. \href{https://arxiv.org/abs/math/0312154}{arXiv:math/0312154}


\end{thebibliography}
\end{document}